\documentclass[11pt,reqno]{amsart}
\usepackage{mathrsfs}
\usepackage{}

\def\subjclass#1{{\renewcommand{\thefootnote}{}%
\footnote{\emph{Mathematics Subject Classification (2020):} #1}}}

\usepackage{amsfonts}
\usepackage{amssymb}

\DeclareMathOperator{\curl}{curl}
\DeclareMathOperator{\divg}{div}

\date{\today}

\hoffset=- 2cm \voffset=- 0cm 
 \theoremstyle{plain}
\newtheorem{Thm}{Theorem}

\newtheorem{Prop}[Thm]{Proposition}
\newtheorem{Rem}[Thm]{Remark}
\newtheorem{Lem}[Thm]{Lemma}

\usepackage{amssymb}
\usepackage{color}

\newcommand {\p}{\partial}

\def\0{\mathbf 0}
\def\a{\alpha}

\def\d{\nabla}

\def\lam{\lambda}

\def\R{\mathbb R}

\def\u{\mathbf u}

\def\v{\vskip}

\errorcontextlines=0 \numberwithin{equation}{section}
\numberwithin{Thm}{section}
\allowdisplaybreaks

\begin{document}
\large
%Topmatter

\title[Liouville type theorem]
{An improved Liouville type theorem for Beltrami flows}

\author{Na Wang}
\author{Zhibing Zhang}

\address{Na Wang: School of Mathematics and Physics, Anhui University of Technology, Ma'anshan 243032, People's Republic of China; }
\email{wn2241750360@163.com}

\address{Zhibing Zhang: School of Mathematics and Physics, Anhui University of Technology, Ma'anshan 243032, People's Republic of China; }
\email{zhibingzhang29@126.com}

\thanks{}



\begin{abstract}
In this note, we improved the Liouville type theorem for the Beltrami flows. Two different methods are used to prove it, one is the monotonicity method, and another is proof by contradiction. The conditions that we proposed on Beltrami flows are significantly weaker than previously known conditions.
\end{abstract}
\maketitle
%end topmatter

\section{Introduction}

Beltrami flow on three dimensional whole space is described by the system
\begin{equation}\label{eq-bel}
  \left\{\begin{aligned}
  &\curl\u\times\u=\0 & \text{in }\mathbb{R}^3,\\
  &\divg\u=0 &\text{in }\mathbb{R}^3,
  \end{aligned}\right.
\end{equation}
where $\u=(u_1,u_2,u_3)$ is a three dimensional vector field.
Since $\curl\u\times\u=(\u\cdot\d)\u-\d|\u|^2/2$, each Beltrami flow gives a special solution to the stationary Euler system. We refer the reader to \cite{Arnold} for the basic properties of Beltrami flows and to \cite{CT20,EP15,EP16,EPT17,EPS18,ELP20} for some recent results. Here we mention that the Beltrami flows are also called force-free magnetic fields in magnetohydrodynamics since the term $\curl\u\times\u$ models the Lorentz force when $\u$ represents the magnetic field, see \cite{Ch,ChKe,Vainshtein}.

In \cite{EP15} and \cite{LLZ2015}, the authors independently constructed Beltrami flows which satisfy $\curl\u=\lam\u$ in $\R^3$ for nonzero constant $\lam$ and fall off as $|\u(x)|<C/|x|$ at infinity. In particular, they are in $L^p(\R^3)$ for all $p>3$. By transforming the Beltrami flow into a suitable tensor form, N. Nadirashvili \cite{Nadi14} proved a Liouville type theorem for $C^1$ smooth Beltrami flows under the condition either $\u\in L^p(\R^3)$, $p\in[2,3]$ or $|\u(x)|=o(|x|^{-1})$ as $|x|\to+\infty$. Subsequently, D. Chae and P. Constantin \cite{CC15} investigated Beltrami solutions of the stationary Euler equations and gave a new and elementary proof to a similar result which partially covers the result of N. Nadirashvili. Later, D. Chae and J. Wolf \cite{CW16} studied weak Beltrami flows, succeeded in covering all cases and got some significant improvements. Before showing their results, we need to introduce their notations of the normal part $\u_N$ and the tangential part $\u_T$ of a vector field $\u$, which are defined respectively as follows:
$$\u_N=\left(\u\cdot\frac{x}{|x|}\right)\frac{x}{|x|},\;\u_T=\frac{x}{|x|}\times\left(\u\times\frac{x}{|x|}\right)=\u-\u_N.$$
Specifically, they proved that if there exists a sequence $R_k\to+\infty$ such that
\begin{equation*}
\int_{\p B_{R_k}(\0)}|\u_T|^2 dS\to 0 \text{ as } k\to +\infty,
\end{equation*}
then the Beltrami flow must be trivial. As immediate consequences, Liouville type theorem holds under one of the following condtions:
\begin{equation}\label{cond1}
\aligned
&(i)\;|\u_T|=o(|x|^{-1}) \text{ as } |x|\to+\infty,\\
&(ii)\;\u_T\in L^p(\mathbb{R}^3,\mathbb{R}^3),\; p\in [2,3],\\
&(iii)\;|\u_T|^2/|x|^\alpha \in L^1(\mathbb{R}^3),\;\alpha\in (-\infty,1].
\endaligned
\end{equation}

Inspired by the assumptions that appear in Liouville type theorems for the stationary Navier-Stokes equations in \cite{Seregin18} and for elliptic equation of divergence form in \cite{BCX08}, and the work of D. Chae and J. Wolf \cite{CW16}, we establish Liouville type results for Beltrami flows under some similar assumptions. Surprisingly, we find that our assumptions are weaker than previous conditions for Beltrami flows.

We say $\mathbf{u}$ is a weak solution to \eqref{eq-bel} if $\mathbf{u}\in W_{loc}^{1,3/2}(\mathbb{R}^3,\mathbb{R}^3)$ satisfies \eqref{eq-bel} locally in the sense of almost everywhere.
Here we mention that $\curl\u\times\u$ belongs to $L^1_{loc}(\mathbb{R}^3,\mathbb{R}^3)$.

Now we state our Liouville type results as follows.

\begin{Thm}\label{liouville}
Let $\mathbf{u}$ be a weak solution to \eqref{eq-bel}. If one of the following conditions holds:
\begin{equation}\label{cond2}
\aligned
&(a) \lim_{R\rightarrow +\infty} \frac{1}{R^{3-p}}\int_{B_{2R}(\mathbf{0})\backslash \overline{B_R(\mathbf{0})}}|\u_T|^p dx=0, \;p\in [2,3],\\
&(b) \lim_{R\rightarrow +\infty} \frac{1}{\ln R}\int_{B_R(\0)\backslash \overline{B_1(\mathbf{0})}}\frac{|\u_T|^2}{|x|}dx=0,\\
&(c) \lim_{R\rightarrow +\infty} \frac{1}{\ln^{1-\beta}R}\int_{B_R(\0)\backslash \overline{B_e(\mathbf{0})}}\frac{|\u_T|^2}{|x|\ln^\beta|x|}dx=0,\;\beta\in (-\infty,1),\\
&(d) \lim_{R\rightarrow +\infty} \frac{1}{\ln\ln R}\int_{B_R(\0)\backslash \overline{B_e(\mathbf{0})}}\frac{|\u_T|^2}{|x|\ln|x|}dx=0,
\endaligned
\end{equation}
then $\u\equiv\0$.
\end{Thm}

\begin{Rem}
Since $|x|^\alpha$ and $R^\alpha$ can be controlled by each other for $x\in B_{2R}(\mathbf{0})\backslash \overline{B_R(\mathbf{0})}$, we see that the condition
$$\lim_{R\rightarrow +\infty} \frac{1}{R}\int_{B_{2R}(\mathbf{0})\backslash \overline{B_R(\mathbf{0})}}|\u_T|^2 dx=0$$
is equivalent to the condition
$$\lim_{R\rightarrow +\infty} \frac{1}{R^{1-\alpha}}\int_{B_{2R}(\mathbf{0})\backslash \overline{B_R(\mathbf{0})}}\frac{|\u_T|^2}{|x|^\alpha} dx=0, \text{ $\alpha\in (-\infty,0)\cup(0,1]$.}$$
\end{Rem}

\begin{Rem}
By a simple calculation, it's not difficult to check that the condition $(i)$ in \eqref{cond1} can be viewed as a special case of anyone in \eqref{cond2}.
Obviously, the condition $(ii)$ in \eqref{cond1} is stronger than the condition $(a)$, and the condition $(iii)$ in \eqref{cond1} is stronger than the conditions $(a)$ and $(b)$. In our setting, the following cases
$$\int_{\mathbb{R}^3}|\u_T|^p dx=+\infty\text{ and }\int_{\mathbb{R}^3}\frac{|\u_T|^2}{|x|^\alpha} dx=+\infty$$
are allowed to happen.
\end{Rem}

\begin{Rem}
For the non-trivial smooth Beltrami flow $\u$ constructed in \cite{EP15} or \cite{LLZ2015}, there exist two positive constants $C_1$ and $C_2$ such that
\begin{equation*}
\aligned
&(1) \lim_{R\rightarrow +\infty} \frac{1}{R}\int_{B_{2R}(\mathbf{0})\backslash \overline{B_R(\mathbf{0})}}|\u_T|^2 dx=C_1,\\
&(2) \lim_{R\rightarrow +\infty} \frac{1}{\ln R}\int_{B_R(\0)\backslash \overline{B_1(\mathbf{0})}}\frac{|\u_T|^2}{|x|}dx=C_2.
\endaligned
\end{equation*}
In fact, since $|\u(x)|\leq C/|x|$ for some positive constant $C$, by \eqref{identity2} we have
$$
\int_{B_R(\mathbf{0})}\frac{|\mathbf{u}_N|^2}{|x|}dx=\frac{1}{2}\int_{ \p B_R(\mathbf{0})}(|\mathbf{u}_T|^2-|\mathbf{u}_N|^2)dS\leq 2\pi C^2,
$$
which follows that
\begin{equation}\label{ineq1}
\frac{1}{R}\int_{B_{2R}(\mathbf{0})\backslash \overline{B_R(\mathbf{0})}}|\u_N|^2 dx\leq 2\int_{B_{2R}(\mathbf{0})\backslash \overline{B_R(\mathbf{0})}}\frac{|\u_N|^2}{|x|} dx\rightarrow 0\text{ as } R\rightarrow +\infty,
\end{equation}
\begin{equation}\label{ineq2}
\frac{1}{\ln R}\int_{B_R(\0)\backslash \overline{B_1(\mathbf{0})}}\frac{|\u_N|^2}{|x|}dx\leq \frac{1}{\ln R}\int_{B_R(\0)}\frac{|\u_N|^2}{|x|}dx\rightarrow 0\text{ as } R\rightarrow +\infty.
\end{equation}
On other hand, by a direct calculation, we find that the functions $E_1(R)$ and $E_2(R)$ given in Proposition \ref{Prop1} are bounded. Considering that $E_1(R)$ and $E_2(R)$ are monotonic, it implies immediately that
\begin{equation}\label{lim1}
\lim_{R\rightarrow +\infty}E_1(R) \text{ and }\lim_{R\rightarrow +\infty}E_2(R) \text{ exist.}
\end{equation}
Combining \eqref{ineq1}, \eqref{ineq2}, \eqref{lim1} and Theorem \ref{liouville}, we see $(1)$ and $(2)$ hold.
\end{Rem}

\begin{Rem}
Ericksen \cite{Ericksen} proved that a smooth unit vector field $\u$ satisfying $\curl\u=\u$ must be equal to
$(\cos x_3,\sin x_3,0)$ in an appropriate coordinate system. For more details on this interesting fact, see
also \cite{Ou}. For this special solution or unit constant vector solution to Beltrami flow, we have
\begin{equation*}
\aligned
&(1) \lim_{R\rightarrow +\infty} \frac{1}{R}\int_{B_{2R}(\mathbf{0})\backslash \overline{B_R(\mathbf{0})}}|\u_T|^2 dx=+\infty,\\
&(2) \lim_{R\rightarrow +\infty} \frac{1}{\ln R}\int_{B_R(\0)\backslash \overline{B_1(\mathbf{0})}}\frac{|\u_T|^2}{|x|}dx=+\infty.
\endaligned
\end{equation*}
Indeed, by \eqref{identity2} we observe that
$$
\int_{ \p B_R(\mathbf{0})}(|\mathbf{u}_T|^2-|\mathbf{u}_N|^2)dS=\frac{1}{R}\int_{ B_R(\mathbf{0})}|\mathbf{u}|^2dx
=\frac{4}{3}\pi R^2.
$$
Hence, it holds that
$$
\frac{1}{R}\int_{B_{2R}(\mathbf{0})\backslash \overline{B_R(\mathbf{0})}}|\u_T|^2 dx\geq E_1(R)= \frac{1}{R}\int_R^{2R}
\int_{\p B_r(\mathbf{0})}(|\mathbf{u}_T|^2-|\mathbf{u}_N|^2)dSdr=\frac{28}{9}\pi R^2,
$$
$$
\frac{1}{\ln R}\int_{B_R(\0)\backslash \overline{B_1(\mathbf{0})}}\frac{|\u_T|^2}{|x|}dx\geq E_2(R)=\frac{1}{\ln R}\int_1^{R}\int_{\p B_r(\mathbf{0})}\frac{|\mathbf{u}_T|^2-|\mathbf{u}_N|^2}{r}dSdr=\frac{2\pi(R^2-1)}{3\ln R}.
$$
\end{Rem}

\v0.1in

\section{Proof of Theorem \ref{liouville}}
Before proving our main results, we need two elementary but useful lemmas. They play a key role in the proof of the main theorem. The first lemma comes from \cite[Lemma 3.2]{ZZ18}. In
the original version, $\alpha\geq0$ and $\u\in H^1(D,\mathbb{R}^3)$. It is not difficult to check that this lemma still holds for $\alpha<0$ and $W^{1,q}(D,\mathbb{R}^3)$ functions, where $q\geq3/2$. Here, the range of $q\geq3/2$ is used to guarantee $\curl\u\times\u\in L^1(D,\mathbb{R}^3)$, $\u\divg\u\in L^1(D,\mathbb{R}^3)$ and $\u\in L^2(\p D,\mathbb{R}^3)$.

\begin{Lem}\label{one-ball}
Let $D$ be a bounded $C^1$ domain in $\mathbb{R}^3$ and $\nu$ denote the unit outer normal on the boundary $\p D$. Assume that $\0\notin \overline{D}$. Let $\mathbf{u}\in W^{1,q}(D,\mathbb{R}^3)$, where $q\geq3/2$. Then it holds that
\begin{equation*}
\aligned
&\int_{D}\curl\mathbf{u}\times\mathbf{u}\cdot \frac{x}{|x|^\alpha}+\left(\mathbf{u}\cdot \frac{x}{|x|^\alpha}\right)\divg \mathbf{u}\,dx\\
=&\int_{D}\frac{1}{|x|^\alpha}\left(\frac{1-\alpha}{2}|\mathbf{u}|^2
+\alpha|\u_N|^2\right)dx+ \int_{\partial D}\left(\mathbf{u}\cdot \frac{x}{|x|^\alpha}\right)(\mathbf{u}\cdot \nu)-\frac{|\mathbf{u}|^2}{2}\left(\frac{x}{|x|^\alpha}\cdot\nu\right)dS\\
=&\int_{D}\frac{1}{|x|^\alpha}\left(\frac{1-\alpha}{2}|\mathbf{u}|^2
+\alpha|\u_N|^2\right)dx+\int_{\partial D}\frac{|\mathbf{u}|^2}{2}\left(\frac{x}{|x|^\alpha}\cdot\nu\right)-\left(\mathbf{u}\times \frac{x}{|x|^\alpha}\right)\cdot(\mathbf{u}\times\nu)dS,
\endaligned
\end{equation*}
where $\alpha$ is a real number.
\end{Lem}

The second lemma can be regarded as a corollary of the first one. D. Chae and J. Wolf \cite{CW16} established similar results for weak Beltrami flows. Since the lower local integrability condition proposed on $\u$, they obtained
a weaker formula
$$
\aligned
\int_{B_R(\mathbf{0})}\frac{|\mathbf{u}_N|^2}{|x|}dx\leq\frac{1}{2R}\int_{ B_R(\mathbf{0})}|\mathbf{u}|^2dx=\frac{1}{2}\int_{ \p B_R(\mathbf{0})}(|\mathbf{u}_T|^2-|\mathbf{u}_N|^2)dS.
\endaligned
$$

\begin{Lem}
Assume $\alpha\leq 1$. Let $\mathbf{u}$ be a weak solution to \eqref{eq-bel}. Then for any $R>0$ it holds that
\begin{equation}\label{identity1}
\aligned
\frac{\alpha-1}{2}\int_{B_R(\mathbf{0})}\frac{|\mathbf{u}|^2
}{|x|^\alpha}dx+\frac{R}{2}\int_{\partial B_R(\mathbf{0})}\frac{|\mathbf{u}|^2}{|x|^\alpha} dS=\alpha\int_{B_R(\mathbf{0})}\frac{|\u_N|^2}{|x|^\alpha}dx+R\int_{\partial B_R(\mathbf{0})}\frac{|\u_N|^2}{|x|^\alpha}dS.
\endaligned
\end{equation}
Moreover, if $\alpha=1$, then for any $R>0$ it holds that
\begin{equation}\label{identity2}
\aligned
\int_{B_R(\mathbf{0})}\frac{|\mathbf{u}_N|^2}{|x|}dx&=\frac{1}{R^2}\int_{ B_R(\mathbf{0})}|x||\mathbf{u}_T|^2dx=\frac{1}{2R}\int_{ B_R(\mathbf{0})}|\mathbf{u}|^2dx\\
&=\frac{1}{2}\int_{ \p B_R(\mathbf{0})}(|\mathbf{u}_T|^2-|\mathbf{u}_N|^2)dS.
\endaligned
\end{equation}
\end{Lem}

\begin{proof}
Let $\eta$ be a smooth function satisfying
$$0\leq\eta\leq 1\text{ in } \mathbb{R}^3,\quad\eta=1 \text{ in } B_R(\0),\quad\eta=0 \text{ outside } B_{R+1}(\0).$$
Since $\mathbf{u}\in W_{loc}^{1,3/2}(\mathbb{R}^3,\mathbb{R}^3)$, by Sobolev embedding and H\"{o}lder inequality we have
$$|\mathbf{u}|^2\in L^1_{loc}(\mathbb{R}^3),\;\nabla|\mathbf{u}|^2=2(u_1\nabla u_1+u_2\nabla u_2+u_3\nabla u_3)\in L^1_{loc}(\mathbb{R}^3,\mathbb{R}^3).$$
Hence, $|\mathbf{u}|^2\eta\in W_0^{1,1}(B_{R+1}(\0)).$
Using Hardy inequality, it follows that
$$\left\|\frac{|\mathbf{u}|^2}{|x|}\right\|_{L^1(B_{R}(\0))}\leq\left\|\frac{|\mathbf{u}|^2\eta}{|x|}\right\|_{L^1(B_{R+1}(\0))}
\leq \frac{1}{2}\left\|\nabla(|\mathbf{u}|^2\eta)\right\|_{L^1(B_{R+1}(\0))},$$
which guarantees
$|\mathbf{u}|^2/|x|^\alpha\in L^1_{loc}(\mathbb{R}^3)$, $\alpha\leq1$.

Let $0<r<R<+\infty$. Set $D=B_R(\mathbf{0})\backslash \overline{B_r(\mathbf{0})}$ in Lemma \ref{one-ball}, then we get
\begin{equation}\label{Rr-ball}
\aligned
&\int_{B_R(\mathbf{0})\backslash \overline{B_r(\mathbf{0})}}\frac{1}{|x|^\alpha}\left(\frac{1-\alpha}{2}|\mathbf{u}|^2
+\alpha|\u_N|^2\right)dx\\
=&-\int_{\partial B_R(\mathbf{0})\cup\partial B_r(\mathbf{0})}\left(\mathbf{u}\cdot \frac{x}{|x|^\alpha}\right)(\mathbf{u}\cdot \nu)-\frac{|\mathbf{u}|^2}{2}\left(\frac{x}{|x|^\alpha}\cdot\nu\right)dS.
\endaligned
\end{equation}
Then we rewrite \eqref{Rr-ball} in the following form:
$$
\aligned
&\int_{B_R(\mathbf{0})}\frac{1}{|x|^\alpha}\left(\frac{1-\alpha}{2}|\mathbf{u}|^2
+\a|\u_N|^2\right)dx+R\int_{\partial B_R(\mathbf{0})}\frac{1}{|x|^\alpha}\left(|\u_N|^2-\frac{|\mathbf{u}|^2}{2}\right)dS\\
=&\int_{B_r(\mathbf{0})}\frac{1}{|x|^\alpha}\left(\frac{1-\alpha}{2}|\mathbf{u}|^2
+\a|\u_N|^2\right)dx+r\int_{\partial B_r(\mathbf{0})}\frac{1}{|x|^\alpha}\left(|\u_N|^2-\frac{|\mathbf{u}|^2}{2}\right)dS.
\endaligned
$$
For simplicity, we denote the above equality by $I_1(R)=I_2(r)+rI_3(r)$, and then integrate this equality with respect to $r$ from $0$ to $t$, where $t>0$. Hence, we obtain
$$
\aligned
|I_1(R)|&\leq \frac{1}{t}\int_0^t|I_2(r)|dr+\int_0^t|I_3(r)|dr\\
&\leq\frac{3|\alpha|+1}{2}\int_{B_t(0)}\frac{|\mathbf{u}|^2}{|x|^\alpha}dx+
\frac{3}{2}\int_{B_t(0)}\frac{|\mathbf{u}|^2}{|x|^\alpha}dx.
\endaligned
$$
Since $|\mathbf{u}|^2/|x|^\alpha\in L^1_{loc}(\mathbb{R}^3)$, letting $t\rightarrow0^+$ we have $I_1(R)=0$, which gives \eqref{identity1}.

If $\alpha=1$, by \eqref{identity1} we get
\begin{equation}\label{eq1}
\int_{B_R(\mathbf{0})}\frac{|\u_N|^2}{|x|} dx+R\int_{\partial B_R(\mathbf{0})}\frac{|\u_N|^2}{|x|}dS=\int_{\partial B_R(\mathbf{0})}\frac{|\mathbf{u}|^2}{2}dS.
\end{equation}
Consequently, we can obtain the following equality (in the sense of weak derivative)
$$
\frac{d}{dR}\left(R\int_{B_R(\mathbf{0})}\frac{|\u_N|^2}{|x|}dx\right)=\frac{d}{dR}\left(\int_{ B_R(\mathbf{0})}\frac{|\mathbf{u}|^2}{2}dx\right),\\
$$
which implies
$$
R\int_{B_R(\mathbf{0})}\frac{|\mathbf{u}_N|^2}{|x|}dx=\frac{1}{2}\int_{ B_R(\mathbf{0})}|\mathbf{u}|^2dx.\\
$$
We rewrite \eqref{eq1} as the following form
$$
\int_{B_R(\mathbf{0})}\frac{|\u_N|^2}{|x|} dx+\frac{R}{2}\int_{\partial B_R(\mathbf{0})}\frac{|\u_N|^2}{|x|} dS=\frac{1}{2}\int_{\partial B_R(\mathbf{0})}|\mathbf{u}_T|^2dS.\\
$$
Then multiplying both sides of the above equality by $2R$, we obtain
$$
2R\int_{B_R(\mathbf{0})}\frac{|\u_N|^2}{|x|} dx+R^{2}\int_{\partial B_R(\mathbf{0})}\frac{|\u_N|^2}{|x|} dS=\int_{\partial B_R(\mathbf{0})}|x||\mathbf{u}_T|^2dS.\\
$$
Immediately, we have the following equality between two weak derivatives
$$
\frac{d}{dR}\left(R^{2}\int_{B_R(\mathbf{0})}\frac{|\u_N|^2}{|x|}dx\right)=\frac{d}{dR}\left(\int_{ B_R(\mathbf{0})}|x||\mathbf{u}_T|^2dx\right).
$$
Therefore, it follows that
$$
R^{2}\int_{B_R(\mathbf{0})}\frac{|\u_N|^2}{|x|}dx=\int_{ B_R(\mathbf{0})}|x||\mathbf{u}_T|^2dx.
$$
\end{proof}

In view of \eqref{identity2}, we see that the quantities
$$\frac{1}{R^2}\int_{ B_R(\mathbf{0})}|x||\mathbf{u}_T|^2dx,\;\frac{1}{R}\int_{ B_R(\mathbf{0})}|\mathbf{u}|^2dx,\;\int_{ \p B_R(\mathbf{0})}(|\mathbf{u}_T|^2-|\mathbf{u}_N|^2)dS$$
are non-decreasing in $R$. Actually, we find more monotonic quantities, which play an important role in the proof of Liouville type theorem.

\begin{Prop}\label{Prop1}
Let $\mathbf{u}$ be a weak solution to \eqref{eq-bel}. Then the quantities
\begin{equation*}
\aligned
&E_1(R)=\frac{1}{R}\int_{B_{2R}(\mathbf{0})\backslash \overline{B_R(\mathbf{0})}}(|\mathbf{u}_T|^2-|\mathbf{u}_N|^2) dx\;(R>0),\\
&E_2(R)=\frac{1}{\ln R}\int_{B_R(\0)\backslash \overline{B_1(\mathbf{0})}}\frac{|\mathbf{u}_T|^2-|\mathbf{u}_N|^2}{|x|}dx\;(R>1),\\
&E_3(R)=\frac{1}{\ln^{1-\beta}R}\int_{B_R(\0)\backslash \overline{B_e(\mathbf{0})}}\frac{|\mathbf{u}_T|^2-|\mathbf{u}_N|^2}{|x|\ln^\beta|x|}dx\;(R>e),\;\beta\in (-\infty,1),\\
&E_4(R)= \frac{1}{\ln\ln R}\int_{B_R(\0)\backslash \overline{B_e(\mathbf{0})}}\frac{|\mathbf{u}_T|^2-|\mathbf{u}_N|^2}{|x|\ln|x|}dx\;(R>e)
\endaligned
\end{equation*}
are non-decreasing in $R$.
\end{Prop}
\begin{proof}
By a direct calculation and using \eqref{identity2}, we obtain
$$
\aligned
\frac{dE_{1}(R)}{dR}&=\frac{2}{R}\int_{\p B_{2R}(\mathbf{0})}(|\mathbf{u}_T|^2-|\mathbf{u}_N|^2)dS-\frac{1}{R}\int_{\p B_R(\mathbf{0})}(|\mathbf{u}_T|^2-|\mathbf{u}_N|^2) dS\\
&\;\;\;\;-\frac{1}{R^2}\int_{B_{2R}(\mathbf{0})\backslash \overline{B_R(\mathbf{0})}}(|\mathbf{u}_T|^2-|\mathbf{u}_N|^2) dx\\
&=\frac{1}{R^2}\left[\int_{B_{2R}(\mathbf{0})\backslash \overline{B_R(\mathbf{0})}}|\mathbf{u}|^2 dx-\int_{B_{2R}(\mathbf{0})\backslash \overline{B_R(\mathbf{0})}}(|\mathbf{u}_T|^2-|\mathbf{u}_N|^2) dx\right]\\
&=\frac{2}{R^2}\int_{B_{2R}(\mathbf{0})\backslash \overline{B_R(\mathbf{0})}}|\mathbf{u}_N|^2 dx\geq0.
\endaligned
$$
Clearly, $E_{1}(R)$ is non-decreasing in $R$.

After a direct calculation, we have
\begin{equation*}
\frac{dE_{2}(R)}{dR}=\frac{1}{(\ln R)^2}\left(\frac{\ln R}{R}\int_{\partial B_{R}(\mathbf{0})}(|\mathbf{u}_T|^2-|\mathbf{u}_N|^2)dS
-\frac{1}{R}\int_{B_{R}(\mathbf{0})\backslash \overline{B_1(\mathbf{0})}}\frac{|\mathbf{u}_T|^2-|\mathbf{u}_N|^2}{|x|} dx \right).
\end{equation*}
Owing to the monotonicity of
$$\int_{\partial B_{r}(\mathbf{0})}(|\mathbf{u}_T|^2-|\mathbf{u}_N|^2)dS,$$
it holds that
$$
\aligned
\int_{B_{R}(\mathbf{0})\backslash \overline{B_1(\mathbf{0})}}\frac{|\mathbf{u}_T|^2-|\mathbf{u}_N|^2}{|x|} dx&=\int_{1}^{R}\int_{\partial B_{r}(\mathbf{0})}\frac{|\mathbf{u}_T|^2-|\mathbf{u}_N|^2}{r} dSdr\\
&\leq\ln R\int_{\partial B_{R}(\mathbf{0})}(|\mathbf{u}_T|^2-|\mathbf{u}_N|^2)dS,
\endaligned
$$
which implies
$$
\frac{dE_{2}(R)}{dR}\geq0.
$$
Therefore $E_2(R)$ is also non-decreasing in $R$.

Similarly to the proof with respect to $E_2(R)$, we can derive that
$$
\aligned
\frac{dE_{3}(R)}{dR}&\geq\frac{1}{(\ln ^{1-\beta}R)^2}\frac{\ln^{1-\beta}R}{R\ln ^\beta R}\int_{\partial B_{R}(\mathbf{0})}(|\mathbf{u}_T|^2-|\mathbf{u}_N|^2)dS\\
&\;\;\;\;-\frac{1}{(\ln ^{1-\beta}R)^2}\frac{1-\beta}{R\ln ^\beta R}\int_{B_R(\0)\backslash \overline{B_e(\mathbf{0})}}\frac{|\mathbf{u}_T|^2-|\mathbf{u}_N|^2}{|x|\ln^\beta|x|}dx\\
&\geq\frac{1}{(\ln ^{1-\beta}R)^2}\frac{1}{R\ln ^\beta R}\int_{\partial B_{R}(\mathbf{0})}(|\mathbf{u}_T|^2-|\mathbf{u}_N|^2)dS\\
&\geq0,
\endaligned
$$
and
$$
\aligned
\frac{dE_{4}(R)}{dR}&\geq\frac{1}{(\ln\ln R)^2}\frac{\ln\ln R}{R\ln R}\int_{\partial B_{R}(\mathbf{0})}(|\mathbf{u}_T|^2-|\mathbf{u}_N|^2)dS\\
&\;\;\;\;-\frac{1}{(\ln\ln R)^2}\frac{1}{R\ln R}\int_{B_R(\0)\backslash \overline{B_e(\mathbf{0})}}\frac{|\mathbf{u}_T|^2-|\mathbf{u}_N|^2}{|x|\ln|x|}dx\\
&\geq0.
\endaligned
$$
Consequently, $E_3(R)$ and $E_4(R)$ are also non-decreasing in $R$.
\end{proof}

Now we are ready to prove our main results.

\begin{proof}[Proof of Theorem \ref{liouville}]
Firstly, we point out that the case $p=2$ of (a) is essential, since the case $p\in(2,3]$ can be converted into the case  $p=2$ by H\"{o}lder inequality.
In fact, for any $p>2$ we have
$$
\aligned
\frac{1}{R}\int_{B_{2R}(\mathbf{0})\backslash \overline{B_R(\mathbf{0})}}|\u_T|^2 dx&\leq\frac{1}{R}\left(\int_{B_{2R}(\mathbf{0})\backslash \overline{B_R(\mathbf{0})}}|\u_T|^p dx\right)^\frac{2}{p}\left(\int_{B_{2R}(\mathbf{0})\backslash \overline{B_R(\mathbf{0})}}dx\right)^{1-\frac{2}{p}}\\
&\leq\left(\frac{28\pi}{3}\right)^{1-\frac{2}{p}}\left(\frac{1}{R^{3-p}}\int_{B_{2R}(\mathbf{0})\backslash \overline{B_R(\mathbf{0})}}|\u_T|^p dx\right)^\frac{2}{p}.
\endaligned
$$
So it is sufficient to deal with the special case $p=2$ of (a) when we prove (a). Assume it holds that
$$\lim_{R\rightarrow +\infty} \frac{1}{R}\int_{B_{2R}(\mathbf{0})\backslash \overline{B_R(\mathbf{0})}}|\u_T|^2 dx=0.$$
From \eqref{identity2}, we see that $E_1(R)$ are nonnegative.
Since
$$0\leq E_1(R)\leq \frac{1}{R}\int_{B_{2R}(\mathbf{0})\backslash \overline{B_R(\mathbf{0})}}|\u_T|^2 dx,$$
we obtain
$$\lim_{R\rightarrow +\infty}E_1(R) =0.$$
By Proposition \ref{Prop1}, we have $E_1(R) =0$ for any $R>0$. Immediately, it follows that
$$
\int_{B_{2R}(\mathbf{0})\backslash \overline{B_R(\mathbf{0})}}(|\mathbf{u}_T|^2-|\mathbf{u}_N|^2) dx=0.
$$
Thanks to the nonnegativity of
$$\int_{\p B_r(\mathbf{0})}(|\mathbf{u}_T|^2-|\mathbf{u}_N|^2) dS,$$
we get
$$\int_{\p B_r(\mathbf{0})}(|\mathbf{u}_T|^2-|\mathbf{u}_N|^2) dS=0\text{ for any $r\in (R,2R)$.}$$
Substituting it into \eqref{identity2}, we get $\u \equiv\0$.

Similarly, we can derive the rest cases (b), (c) and (d).

\end{proof}

\begin{proof}[Another Proof of Theorem \ref{liouville}]

We first prove case (a). We claim that $\mathbf{u}_N \equiv0$.
If not, then we can choose $R_0$ sufficiently large such that\\
\begin{equation}\label{theta}
\aligned
\int_{B_{R_0}(\mathbf{0})} \frac{|\mathbf{u}_N|^2}{|x|} dx =\theta >0.
\endaligned
\end{equation}
Using \eqref{identity2} and applying H\"{o}lder inequality, we find that
\begin{equation*}
\aligned
\int_{B_{R}(\mathbf{0})} \frac{|\mathbf{u}_N|^2}{|x|} dx&\leq \frac{1}{2}\int_{ \p B_R(\mathbf{0})}|\mathbf{u}_T|^2dS\leq \frac{1}{2}(4\pi R^2)^{1-\frac{2}{p}}\left(\int_{\p B_R(\mathbf{0})}|\mathbf{u}_T|^p dS \right)^\frac{2}{p}.
\endaligned
\end{equation*}
Thus it follows that
\begin{equation}\label{eq2}
\aligned
 \frac{(2\theta)^\frac{p}{2}}{R^{p-2}}\leq (4\pi)^{\frac{p}{2}-1}\int_{ \p B_R(\mathbf{0})}|\mathbf{u}_T|^pdS\text{   for all $R\geq R_0$.}
\endaligned
\end{equation}
Integrating the above inequality with respect to $R$ from $R_1$ to $2R_1$, where $R_1>R_0$, we obtain
\begin{equation*}
\aligned
C(p)(2\theta)^\frac{p}{2}{R_1}^{3-p}\leq (4\pi)^{\frac{p}{2}-1}\int_{B_{2R_1}(\mathbf{0})\backslash \overline{B_{R_1}(\mathbf{0})}}|\mathbf{u}_T|^p dx,
\endaligned
\end{equation*}
where $C(p)=(2^{3-p}-1)/(3-p)$ if $p\in [2,3)$, $C(p)=\ln 2$ if $p=3$.
Then we get
\begin{equation*}
\aligned
C(p)(2\theta)^\frac{p}{2}\leq (4\pi)^{\frac{p}{2}-1}\frac{1}{{R_1}^{3-p}}\int_{B_{2R_1}(\mathbf{0})\backslash \overline{B_{R_1}(\mathbf{0})}}|\mathbf{u}_T|^p dx\rightarrow 0 \text{ as $R_{1}\rightarrow +\infty$.}
\endaligned
\end{equation*}
Therefore $\theta \leq 0$, which is a contradiction. So $\mathbf{u}_N \equiv0$. Substituting it into \eqref{identity2}, we get $\u \equiv\0$.

Next we prove case (b), (c) and (d). We claim that $\mathbf{u}_N \equiv0$. If not, then \eqref{theta} holds.
Using \eqref{identity2} and \eqref{theta}, we get
\begin{equation*}
\aligned
&\frac{\theta}{R} \leq\frac{1}{R} \int_{B_{R}(\mathbf{0})} \frac{|\mathbf{u}_N|^2}{|x|} dx\leq \frac{1}{2}\int_{ \p B_R(\mathbf{0})}\frac{|\mathbf{u}_T|^2}{|x|}dS,\\
&\frac{\theta}{R\ln^\beta R} \leq\frac{1}{R\ln^\beta R} \int_{B_{R}(\mathbf{0})} \frac{|\mathbf{u}_N|^2}{|x|} dx\leq \frac{1}{2}\int_{ \p B_R(\mathbf{0})}\frac{|\mathbf{u}_T|^2}{|x|\ln^\beta |x|}dS,
\endaligned
\end{equation*}
where $R\geq R_0$. Let $R_{0}< R_{1}$. Integrating the above inequalities with respect to $R$ from $R_{0}$ to $R_{1}$, respectively, it follows that\\
\begin{equation*}
\aligned
&\theta (\ln R_1-\ln R_0) \leq \frac{1}{2}\int_{B_{R_1}(\mathbf{0})\backslash \overline{B_{R_0}(\mathbf{0})}}\frac{|\mathbf{u}_T|^2}{|x|}dx,\\
&\frac{\theta\left(\ln^{1-\beta} R_1-\ln^{1-\beta} R_0\right)}{1-\beta}\leq \frac{1}{2}\int_{B_{R_1}(\mathbf{0})\backslash \overline{B_{R_0}(\mathbf{0})}}\frac{|\mathbf{u}_T|^2}{|x|\ln^\beta|x|}dx,\text{ if } \beta<1,\\
&\theta\left(\ln \ln R_1-\ln \ln R_0\right)\leq \frac{1}{2}\int_{B_{R_1}(\mathbf{0})\backslash \overline{B_{R_0}(\mathbf{0})}}\frac{|\mathbf{u}_T|^2}{|x|\ln |x|}dx,\text{ if } \beta=1.
\endaligned
\end{equation*}
Thus we have
\begin{equation*}
\aligned
&\theta\leq \frac{1}{2}\cdot\frac{\ln R_1}{\ln R_1-\ln R_0}\cdot\frac{1}{\ln R_1}\int_{B_{R_1}(\mathbf{0})\backslash \overline{B_{1}(\mathbf{0})}}\frac{|\mathbf{u}_T|^2}{|x|}dx,\\
&\theta\leq\frac{1}{2}\cdot\frac{(1-\beta)\ln^{1-\beta}R_1}{\left(\ln^{1-\beta} R_1-\ln^{1-\beta}R_0\right)}\cdot\frac{1}{\ln^{1-\beta}R_1}\int_{B_{R_1}(\mathbf{0})\backslash \overline{B_e(\mathbf{0})}}\frac{|\mathbf{u}_T|^2}{|x|\ln^\beta|x|}dx,\text{ if } \beta<1,\\
&\theta\leq \frac{1}{2}\cdot\frac{\ln \ln R_1}{\left(\ln \ln R_1-\ln \ln R_0\right)}\cdot\frac{1}{\ln \ln R_1}\int_{B_{R_1}(\mathbf{0})\backslash \overline{B_e(\mathbf{0})}}\frac{|\mathbf{u}_T|^2}{|x|\ln |x|}dx,\text{ if } \beta=1.
\endaligned
\end{equation*}
Letting $R_{1}\rightarrow +\infty$, we get a contradiction for each case. Routinely, we obtain $\u \equiv\0$.
\end{proof}

\begin{Rem}
Owing to the monotonicity of
$$\frac{1}{R^2}\int_{ B_R(\mathbf{0})}|x||\mathbf{u}_T|^2dx\text{  and  }\frac{1}{R}\int_{ B_R(\mathbf{0})}|\mathbf{u}|^2dx,$$
\eqref{theta} can be replaced by
$$
\aligned
&\text{if $\u_T\neq\0$, then there exists $R_0>0$ such that }\frac{1}{R_0^2}\int_{ B_{R_0}(\mathbf{0})}|x||\mathbf{u}_T|^2dx=\theta>0,\text{  or  }\\
&\text{if $\u\neq\0$, then there exists $R_0>0$ such that }\frac{1}{R_0}\int_{ B_{R_0}(\mathbf{0})}|\mathbf{u}|^2dx=\theta>0.
\endaligned
$$
\end{Rem}

\begin{Rem}
From the above two proofs, we see that the term $|\u_T|^2$ in \eqref{cond2} can be replaced by $|\u_T|^2-|\u_N|^2$.
\end{Rem}

\subsection*{Acknowledgements.}
Zhang would like to thank Professor Dong Ye for helpful discussions. The referees are thanked for valuable comments and suggestions that helped to improve the manuscript. This work was supported by the National Natural Science Foundation of China Grant No. 11901003 and Anhui Provincial Natural Science Foundation Grant No. 1908085QA28.

\end {document}